\documentclass{article}

\usepackage{amsmath,amssymb,amsthm}
\newtheorem{lemma}{Lemma}
\newtheorem{proposition}{Proposition}
\newtheorem{theorem}{Theorem}

\usepackage[backend=biber, style=numeric]{biblatex} 
\title{Computing lower expectations with respect to total variation distance and chi-squared divergence balls}
\author{Jasper De Bock}
\date{}
\begin{document}
\maketitle

\begin{abstract}
We derive closed form expressions for the lower expectations that correspond to total variation distance and chi-squared divergence balls around a probability mass function over a finite set.
\end{abstract}

\section{Introduction}
In the context of our work on robustness quantification for probabilistic classifiers~\cite{detavernier2025robustness}, we had to evaluate the lower expectations that correspond to total variation distance~\cite{gibbsdistances} and chi-squared divergence \cite{gibbsdistances} balls around a categorical/multinomial distribution. To our surprise, while it seems to us that this is a task that is essentially well understood from Distributionally Robust Optimization (DRO), we were unable to find closed form expressions for these lower expectations in the literature. This note aims to remedy this situation.

\section{General setting}

Let $\mathcal{X}$ be a finite non-empty set, with $|\mathcal{X}|=n$, let $\Sigma_\mathcal{X}$ be the set of all probability mass functions over $\mathcal{X}$, and let $\mathbb{R}^\mathcal{X}$ be the set of all real-valued functions on $\mathcal{X}$.

For any $f\in\mathbb{R}^\mathcal{X}$ and any $p\in\Sigma_\mathcal{X}$, we write $E_p(f)$ for the expectation of $f$ with respect to $p$, which is defined as
\begin{equation*}
    E_p(f):=\sum_{x\in\mathcal{X}}p(x)f(x).
\end{equation*}
For any non-empty set $\mathcal{M}\subseteq\Sigma_\mathcal{X}$, we write $\underline{E}_\mathcal{M}(f)$ for the lower expectation of $f$ with respect to $\mathcal{M}$, which is defined as \begin{equation*}
    \underline{E}_\mathcal{M}(f):=\inf_{p\in\mathcal{M}}E_p(f).
\end{equation*}
Concretely, we are interested in computing $\underline{E}_\mathcal{M}(f)$ for sets $\mathcal{M}$ of the form
\begin{equation*}
    B_d(p,\delta)=\{q\in\Sigma_\mathcal{X}:d(q,p)\leq\delta\},
\end{equation*}where $p\in\Sigma_\mathcal{X}$ is a given probability mass function, $d:\Sigma_\mathcal{X}\times\Sigma_\mathcal{X}\to\mathbb{R}_{\geq0}$ is a given distance/divergence, and $\delta\in\mathbb{R}_{\geq0}$ is a given radius. More specifically, we focus on the cases where $d$ is the total variation distance or the $\chi^2$-divergence.

\section{Total variation balls}\label{sec:tv}

The first case that we consider is when $d$ is the total variation distance, which is defined for all $q\in\Sigma_\mathcal{X}$ by
\begin{equation*}
    d_\mathrm{TV}(q,p):=\frac{1}{2}\sum_{x\in\mathcal{X}}|q(x)-p(x)|.
\end{equation*}
In this case, for any $\delta\in[0,1]$, we write $B_\mathrm{TV}(p,\delta)$ instead of $B_{d_\mathrm{TV}}(p,\delta)$.

For any given $f\in\mathbb{R}^\mathcal{X}$, our aim is to compute $\underline{E}_{B_\mathrm{TV}(p,\delta)}(f)$, which we will henceforth denote by $\underline{E}^\mathrm{TV}_{p,\delta}(f)$ for the sake of brevity.

The first step in computing $\underline{E}^\mathrm{TV}_{p,\delta}(f)$ is to sort the elements of $\mathcal{X}=\{x_1,\ldots,x_n\}$ in such a way that $f(x_1)\leq f(x_2)\leq\cdots\leq f(x_n)$. We will henceforth assume that $\mathcal{X}$ is ordered in this way.

Now let $r\in\{1,\ldots,n\}$ be the smallest index such that $\delta\geq\sum_{i=r+1}^{n} p(x_i)$. Note that such an index must exist, since $\sum_{i=n+1}^n p(x_i)=0$ and $\delta\in[0,1]$.
The following result then provides the desired closed form expression.

\begin{theorem}\label{thm:lowerexpectationtv}
Let $r$ be defined as above. Then
\begin{equation*}
    \underline{E}^\mathrm{TV}_{p,\delta}(f)
    =
    E_{q_\mathrm{TV}}(f)
\end{equation*}
where, if $r>1$, $q_\mathrm{TV}\in\Sigma_\mathcal{X}$ is given by
\begin{equation*}
    q_\mathrm{TV}(x_i):=
    \begin{cases}
        p(x_1)+\delta & \text{if } i=1,\\
        p(x_i) & \text{if } i\in\{2,\ldots,r-1\},\\
        \sum_{i=r}^{n} p(x_i)-\delta & \text{if } i=r,\\
        0 & \text{if } i\in\{r+1,\ldots,n\}.
    \end{cases}
\end{equation*}
and, if $r=1$, $q_\mathrm{TV}\in\Sigma_\mathcal{X}$ is given by
\begin{equation*}
    q_\mathrm{TV}(x_i):=
    \begin{cases}
        1 & \text{if } i=1,\\
        0 & \text{if } i\in\{2,\ldots,n\}.
    \end{cases}
\end{equation*}
\end{theorem}
Intuitively, this corresponds to constructing $q_\mathrm{TV}$ starting from $p$, by moving as much mass as possible from the elements of $\mathcal{X}$ with the highest $f$-values to the element of $\mathcal{X}$ with the lowest $f$-value, until we have moved a total mass of $\delta$ or, if $r=1$, have moved all available mass.

\section{chi-squared divergence balls}\label{sec:chi2}

The second case that we consider is when $d$ is the chi-squared divergence, often denoted as $\chi^2$-divergence, which, if $p(x_i)>0$ for all $i\in\{1,\ldots,n\}$, is defined for all $q\in\mathbb{R}^\mathcal{X}$ by
\begin{equation*}
    d_{\chi^2}(q,p):=\sum_{x\in\mathcal{X}}\frac{(q(x)-p(x))^2}{p(x)}
    =\sum_{x\in\mathcal{X}}p(x)\left(\frac{q(x)}{p(x)}-1\right)^2.
\end{equation*}
In this case, for any $\delta\geq0$, we write $B_{\chi^2}(p,\delta)$ instead of $B_{d_{\chi^2}}(p,\delta)$.

For any given $f\in\mathbb{R}^\mathcal{X}$, our aim is to compute $\smash{\underline{E}_{B_{\chi^2}(p,\delta)}(f)}$, which we will henceforth denote by $\underline{E}^{\chi^2}_{p,\delta}(f)$ for the sake of brevity.

The first step is again to sort the elements of $\mathcal{X}=\{x_1,\ldots,x_n\}$ in such a way that $f(x_1)\leq f(x_2)\leq\cdots\leq f(x_n)$. Let $\ell$ be the largest index such that $f(x_\ell)=f(x_1)$. Next, for any $k\in\{\ell,\ldots,n\}$, we define
\begin{equation*}
    m_k:=\sum_{i=1}^k p(x_i),~~
\mu_k:=\frac{1}{m_k}\sum_{i=1}^k p(x_i)f(x_i)
\text{ and }\sigma_k^2:=\frac{1}{m_k}\sum_{i=1}^k p(x_i)(f(x_i)-\mu_k)^2
\end{equation*}
We also let $\delta_\ell:=+\infty$ and, for any $k\in\{\ell+1,\ldots,n\}$, let
\begin{equation*}
    \delta_k:=\frac{1}{m_k}\left(\frac{\sigma_k^2}{(f(x_k)-\mu_k)^2}+1-m_k\right).
\end{equation*}
These $\delta_k$ are the critical values of $\delta$ at which the lower expectation $\underline{E}^{\chi^2}_{p,\delta}(f)$ changes its form, as we will see below. We start out by showing that they are nicely ordered.
\begin{proposition}\label{prop:orderingdeltak}
The $\delta_k$ are well-defined, positive and non-increasing in $k$:
\begin{equation*}
    +\infty=\delta_\ell>\delta_{\ell+1}\geq\cdots\geq\delta_{n-1}\geq\delta_n>0.
\end{equation*}
\end{proposition}

Now let $r\in\{\ell,\ldots,n\}$ be the largest index such that $\delta_r>\delta$. Due to the previous proposition, this means that $r$ is the unique value in $\{\ell,\ldots,n\}$ such that
\begin{equation*}
\delta_\ell\geq\cdots\geq\delta_r>\delta\geq\delta_{r+1}\geq\cdots\geq\delta_n\geq0
\end{equation*}

The following theorem then gives us the desired closed form expression for the lower expectation that we are after.

\begin{theorem}\label{thm:lowerexpectationchi2}
With $r\in\{\ell,\ldots,n\}$ defined as above, we have that
\begin{equation}\label{eq:lowerexpectationchi2}
    \underline{E}^{\chi^2}_{p,\delta}(f)
    =
    \begin{cases}
    \mu_r-\sigma_r\sqrt{m_r\delta-(1-m_r)} & \text{if } r>\ell,\\
    f(x_1) & \text{if } r=\ell.
    \end{cases}
\end{equation}
\end{theorem}

If $\vert\mathcal{X}\vert=3$, Equation~\eqref{eq:lowerexpectationchi2} for example results in the following case distinction:
\begin{equation}\label{eq:lowerexpectationchi2three}
    \underline{E}^{\chi^2}_{p,\delta}(f)=
    \begin{cases}
        \mu_3-\sigma_3\sqrt{\delta} & \text{if } \delta<\delta_3,\\
        \mu_2-\sigma_2\sqrt{m_2\delta-(1-m_2)} & \text{if } \delta_3\leq\delta<\delta_2,\\
        f(x_1) & \text{if } \delta\geq\delta_2.
    \end{cases}
\end{equation}
If $\vert\mathcal{X}\vert=2$, the resulting case distinction can be simplified even further, resulting in the following simple expression:
\begin{equation}\label{eq:lowerexpectationchi2two}
    \underline{E}^{\chi^2}_{p,\delta}(f)=
    \begin{cases}
        \displaystyle\sum_{i=1}^2p(x_i)f(x_i)-\sqrt{\delta p(x_1)p(x_2)}\vert f(x_2)-f(x_1)\vert & \text{if } \delta<\frac{p(x_2)}{p(x_1)},\\
        f(x_1) & \text{if } \delta\geq\frac{p(x_2)}{p(x_1)}.
    \end{cases}
\end{equation}

\printbibliography

\appendix

\section{Proofs for the results in Section~\ref{sec:tv}}

\begin{proof}[Proof of Theorem~\ref{thm:lowerexpectationtv}]
If $r=1$, then $\delta\geq\sum_{i=2}^n p(x_i)$, which implies that
\begin{align*}
d_{\mathrm{TV}}(q_\mathrm{TV},p)
&=\frac{1}{2}\sum_{i=1}^n|q_\mathrm{TV}(x_i)-p(x_i)|\\
&=\frac{1}{2}\left(1-p(x_1)+\sum_{i=2}^n p(x_i)\right)
=\sum_{i=2}^n p(x_i)\leq\delta.
\end{align*}
The result then follows because $\min f\leq\underline{E}^\mathrm{TV}_{p,\delta}(f)\leq E_{q_\mathrm{TV}}(f)=f(x_1)=\min f$. So we can now assume that $r>1$, which implies that 
\begin{equation*}
\sum_{i=r+1}^n p(x_i)\leq\delta<\sum_{i=r}^{n} p(x_i)\leq\sum_{i=2}^n p(x_i)=1-p(x_1).
\end{equation*}

Since $B_\mathrm{TV}(p,\delta)$ is a non-empty closed subset of $\Sigma_\mathcal{X}$, the minimum $\underline{E}^\mathrm{TV}_{p,\delta}(f)$ is well-defined and there is at least one $q\in B_\mathrm{TV}(p,\delta)$ such that $\underline{E}^\mathrm{TV}_{p,\delta}(f)=E_q(f)$. Since the set of all $q$ that attain this minimum is also closed, we can choose $q^*$ to be one among them that maximizes $q(x_1)$.

We first show that $q^*(x_1)\geq p(x_1)$. Assume \emph{ex absurdo} that $q^*(x_1)<p(x_1)$. Since $q^*$ and $p$ are both probability mass functions, this implies that there is some $i\in\{2,\ldots,n\}$ such that $q^*(x_i)>p(x_i)$. Consider any such index. Then we can construct a new probability mass function $q'$ by mass $\alpha=\min\{p(x_1)-q^*(x_1),q^*(x_{i})-p(x_{i})\}>0$ from $x_{i}$ to $x_1$. This mass transfer does not alter the total variation distance to $p$, since we are moving the same mass $\alpha$ in opposite directions, and the expectation of $f$ does not increase since $f(x_1)\leq f(x_i)$. This contradicts the definition of $q^*$, which was chosen to have the highest value of $q(x_1)$ among all $q$ such that $\underline{E}^\mathrm{TV}_{p,\delta}(f)=E_q(f)$. So we must have that $q^*(x_1)\geq p(x_1)$, as claimed.

Next, we show that $q^*(x_i)\leq p(x_i)$ for all $i\in\{2,\ldots,n\}$. Assume \emph{ex absurdo} that there is some $i\in\{2,\ldots,n\}$ such that $q^*(x_i)>p(x_i)$. Consider any such index. Then we can construct a new probability mass function $q'$ by moving mass $\alpha=q^*(x_i)-p(x_i)>0$ from $x_i$ to $x_1$. This mass transfer does not alter the total variation distance to $p$, since we are moving the same mass $\alpha$ in opposite directions, and the expectation of $f$ does not increase since $f(x_1)\leq f(x_i)$. This contradicts the definition of $q^*$, which was chosen to have the highest value of $q(x_1)$ among all $q$ such that $\underline{E}^\mathrm{TV}_{p,\delta}(f)=E_q(f)$. So we must have that $q^*(x_i)\leq p(x_i)$ for all $i\in\{2,\ldots,n\}$, as claimed.

As a next step, we show that $q^*(x_1)=p(x_1)+\delta$. Assume \emph{ex absurdo} that $p(x_1)\leq q^*(x_1)<p(x_1)+\delta$. Since $\delta<1-p(x_1)$, this implies that $q^*(x_1)<1$. So there is some $i\in\{2,\ldots,n\}$ such that $0<q^*(x_i)\leq p(x_i)$. Consider any such $i$. Then we can construct a new probability mass function $q'$ by moving mass $\alpha=q^*(x_i)>0$ from $x_i$ to $x_1$. This mass transfer again does not alter the total variation distance to $p$, nor does it increase expectation of $f$, once more contradicting the definition of $q^*$. So we must have that $q^*(x_1)=p(x_1)+\delta$, as claimed.

We conclude that $q*$ belongs to the set
\begin{equation*}
\mathcal{Q}^*:=\{q\in\Sigma_\mathcal{X}: q(x_1)=p(x_1)+\delta\text{ and }q(x_i)\leq p(x_i)\text{ for all }i\in\{2,\ldots,n\}\}.
\end{equation*}
Since $d_\mathrm{TV}(q,p)=\delta$ for all $q\in\mathcal{Q}^*$ and therefore $\mathcal{Q}^*\subseteq B_\mathrm{TV}(p,\delta)$, and since $q^*$ was such that $\underline{E}^\mathrm{TV}_{p,\delta}(f)=E_{q^*}(f)$, this implies that $\underline{E}^\mathrm{TV}_{p,\delta}(f)=\min_{q\in\mathcal{Q}^*} E_q(f)$. Since the set of all $q\in\mathcal{Q}^*$ that attain the minimum $\min_{q\in\mathcal{Q}^*} E_q(f)$ is closed, we can choose $q^{**}$ to be one among them that maximizes $\sum_{i=2}^r q(x_i)$.

We now show that $q^{**}(x_i)=0$ for all $i\in\{r+1,\ldots,n\}$. Assume \emph{ex absurdo} that there is some $i\in\{r+1,\ldots,n\}$ such that $q^{**}(x_i)>0$. Consider any such index $i$. Then
\begin{align*}
\sum_{i=2}^r q^{**}(x_i)
&=1-q^{**}(x_1)-\sum_{i=r+1}^n q^{**}(x_i)\\
&<1-q^{**}(x_1)=1-p(x_1)-\delta\\
&\leq1-p(x_1)-\sum_{i=r+1}^n p(x_i)=\sum_{i=2}^r p(x_i).
\end{align*}
Since $q^{**}(x_i)\leq p(x_i)$ for all $i\in\{2,\ldots,n\}$, this implies that there is some $j\in\{2,\ldots,r\}$ such that $q^{**}(x_j)<p(x_j)$. Consider any such index $j$. Then we can construct a new probability mass function $q'$ by moving mass $\alpha=\min\{p(x_j)-q^{**}(x_j),q^{**}(x_i)\}>0$ from $x_i$ to $x_j$. This mass function $q'$ still belongs to $\mathcal{Q}^*$, and the expectation of $f$ does not increase since $f(x_j)\leq f(x_i)$. This however contradicts the definition of $q^{**}$, which was chosen to have the highest value of $\sum_{i=2}^r q(x_i)$ among all $q\in\mathcal{Q}^*$ that attain the minimum $\min_{q\in\mathcal{Q}^*} E_q(f)$. So we must have that $q^{**}(x_i)=0$ for all $i\in\{r+1,\ldots,n\}$, as claimed.

So we find that $q^{**}$ belongs to the set
\begin{multline*}
\mathcal{Q}^{**}:=\{q\in\Sigma_\mathcal{X}: q(x_1)=p(x_1)+\delta,~q(x_i)\leq p(x_i)\text{ for all }i\in\{2,\ldots,r\}\\\text{ and }q(x_i)=0\text{ for all }i\in\{r+1,\ldots,n\}\}.
\end{multline*}
Since $\min_{q\in\mathcal{Q}^*} E_q(f)=E_{q^{**}}(f)$ and $\mathcal{Q}^{**}\subseteq\mathcal{Q}^*$, this implies that
\begin{equation*}
\min_{q\in\mathcal{Q}^{**}} E_q(f)=\min_{q\in\mathcal{Q}^*}E_q(f)=\underline{E}^\mathrm{TV}_{p,\delta}(f).
\end{equation*}
Since the set of all $q\in\mathcal{Q}^{**}$ that attain the minimum $\min_{q\in\mathcal{Q}^{**}} E_q(f)$ is closed, we can choose $q_\mathrm{TV}$ to be one among them that minimizes $q_\mathrm{TV}(x_r)$.

Finally, we show that $q_\mathrm{TV}(x_i)=p(x_i)$ for all $i\in\{2,\ldots,r-1\}$. Assume \emph{ex absurdo} that there is some $i\in\{2,\ldots,r-1\}$ such that $q_\mathrm{TV}(x_i)<p(x_i)$. Consider any such index $i$. Since $q_\mathrm{TV}(x_i)\leq p(x_i)$ for all $i\in\{2,\ldots,r\}$, this implies that $\sum_{i=2}^{r-1} q_\mathrm{TV}(x_i)<\sum_{i=2}^{r-1} p(x_i)$. Since $q_\mathrm{TV}(x_1)=p(x_1)+\delta$ and $q_\mathrm{TV}(x_i)=0$ for all $i\in\{r+1,\ldots,n\}$, this implies that
\begin{align*}
q_{\mathrm{TV}}(x_r)
&=1-q_{\mathrm{TV}}(x_1)-\sum_{i=2}^{r-1} q_{\mathrm{TV}}(x_i)\\
&>1-p(x_1)-\delta-\sum_{i=2}^{r-1} p(x_i)\\
&\geq1-p(x_1)-\sum_{i=r+1}^n-\sum_{i=2}^{r-1}p(x_i)=p(x_r)\geq0.
\end{align*}
So $q_\mathrm{TV}(x_i)<p(x_i)$ but $q_{\mathrm{TV}}(x_r)>0$. Then we can construct a new probability mass function $q'$ by moving mass $\alpha=\min\{p(x_i)-q_\mathrm{TV}(x_i),q_\mathrm{TV}(x_r)\}>0$ from $x_r$ to $x_i$. This mass function $q'$ still belongs to $\mathcal{Q}^{**}$, and the expectation of $f$ does not increase since $f(x_i)\leq f(x_r)$. This however contradicts the definition of $q_\mathrm{TV}$, which was chosen to have the lowest value of $q(x_r)$ among all $q\in\mathcal{Q}^{**}$ that attain the minimum $\min_{q\in\mathcal{Q}^{**}} E_q(f)$. So we must have that $q_\mathrm{TV}(x_i)=p(x_i)$ for all $i\in\{2,\ldots,r-1\}$, as claimed.

So we conclude that our chosen $q_\mathrm{TV}\in\mathcal{Q}^{**}\subseteq\mathcal{Q}^*\subseteq B_\mathrm{TV}(p,\delta)$ such that $\underline{E}^\mathrm{TV}_{p,\delta}(f)=E_{q_\mathrm{TV}}(f)$ satisfies $q_\mathrm{TV}(x_1)=p(x_1)+\delta$, $q_\mathrm{TV}(x_i)=p(x_i)$ for all $i\in\{2,\ldots,r-1\}$ and $q_\mathrm{TV}(x_i)=0$ for all $i\in\{r+1,\ldots,n\}$, and therefore also
\begin{equation*}
q_\mathrm{TV}(x_r)=1-q_\mathrm{TV}(x_1)-\sum_{i=2}^{r-1} q_\mathrm{TV}(x_i)=1-p(x_1)-\delta-\sum_{i=2}^{r-1} p(x_i)=\sum_{i=r}^{n} p(x_i)-\delta>0,
\end{equation*}
as required.
\end{proof}

\section{Proofs for the results in Section~\ref{sec:chi2}}

Our proof for the results in Section~\ref{sec:chi2} make use of a number of intermediate results, which we will state and prove in the form of lemmas. For notational convenience, for all $k\in\{1,\ldots,n\}$, we also let
\begin{equation*}
\mathcal{Q}_k:=\{q\in\mathbb{R}^{\mathcal{X}}
\colon
\sum_{i=1}^n q(x_i)=1, d_{\chi^2}(q,p)\leq\delta,~q(x_i)=0\text{ for all }i>k\}.
\end{equation*}

\begin{lemma}\label{lemma:convexitychi2}
For any $k\in\{1,\ldots,n\}$, the set $\mathcal{Q}_k$ is convex.
\end{lemma}
\begin{proof}
Consider any $q_1,q_2\in\mathcal{Q}_k$ and any $\lambda\in[0,1]$. Then $q=\lambda q_1+(1-\lambda)q_2$ satisfies the $\chi^2$-divergence constraint because
\begin{align*}
d_{\chi^2}(q,p)
&=\sum_{i=1}^n p(x_i)\left(\frac{\lambda q_1(x_i)+(1-\lambda)q_2(x_i)}{p(x_i)}-1\right)^2\\
&=\sum_{i=1}^n p(x_i)\left(\lambda\frac{q_1(x_i)}{p(x_i)}+(1-\lambda)\frac{q_2(x_i)}{p(x_i)}-1\right)^2\\
&\leq\sum_{i=1}^n p(x_i)\left(\lambda\left(\frac{q_1(x_i)}{p(x_i)}-1\right)^2+(1-\lambda)\left(\frac{q_2(x_i)}{p(x_i)}-1\right)^2\right)\\
&=\lambda d_{\chi^2}(q_1,p)+(1-\lambda)d_{\chi^2}(q_2,p)\leq\delta,
\end{align*}
where we have used the convexity of the function $t\mapsto (t-1)^2$ in the inequality (which holds because the second derivative of this function is positive). Since the other properties are trivially preserved under convex combinations, this implies that $q\in\mathcal{Q}_k$.
\end{proof}

\begin{lemma}\label{lemma:lowerexpectationchi2}
For any $k\in\{\ell+1,\ldots,n\}$, the following are equivalent:
\begin{equation*}
\mathcal{Q}_k\neq\emptyset
\Leftrightarrow
m_k\delta-(1-m_k)\geq0
\Leftrightarrow
\delta\geq\frac{1-m_k}{m_k}.
\end{equation*}
Whenever either of these is true, we have that
\begin{equation}\label{eq:lemma:lowerexpectationchi2}
    \min_{q\in\mathcal{Q}_k} E_q(f)=\mu_k-\sigma_k\sqrt{m_k\delta-(1-m_k)}.
\end{equation}
The minimum is then furthermore attained uniquely for $q_k\in\mathbb{R}^\mathcal{X}$ with $q_k(x_i)=0$ for all $i>k$ and, for all $i\in\{1,\ldots,k\}$,
\begin{equation}\label{eq:lemma:lowerexpectationchi2:qk}
    q_k(x_i)
    =\frac{p(x_i)}{m_k}\left(1-\frac{f(x_i)-\mu_k}{\sigma_k}\sqrt{m_k\delta-(1-m_k)}\right).
\end{equation}
Finally, $q_k(x_k)>0$ (and then also $q_k(x_i)>0$ for all $i\in\{1,\ldots,k\}$) if and only if $\delta<\delta_k$, and $q_k(x_k)=0$ (and then also $q_k(x_i)\geq0$ for all $i\in\{1,\ldots,k\}$) if and only if $\delta=\delta_k$
\end{lemma}
\begin{proof}
Let $\mathcal{X}_k:=\{x_1,\ldots,x_k\}$ and consider the probability mass function $p_k$ on $\mathcal{X}_k$ defined by $p_k(x_i)=\frac{p(x_i)}{m_k}>0$ for all $i\in\{1,\ldots,k\}$. For all $g\in\mathbb{R}^{\mathcal{X}_k}$, we write
\begin{equation*}
E_k(g)=\sum_{i=1}^k p_k(x_i)g(x_i)
\text{ and }
\mathrm{Var}_k(g)=E_k(g-E_k(g))^2
\end{equation*}
for the expectation and variance of $g$ with respect to $p_k$. In particular, if we use $f_k$ to denote the restriction of $f$ to $\mathcal{X}_k$, we then have that $\mu_k=E_k(f_k)$ and $\sigma_k^2=\mathrm{Var}_k(f_k)$.

Consider any $q\in\mathbb{R}^\mathcal{X}$ such that $\sum_{i=1}^kq(x_i)=1$ and $q(x_i)=0$ for all $i>k$ and let $u,y\in\mathbb{R}^{\mathcal{X}_k}$ be defined by $u(x_i)=\frac{q(x_i)}{p(x_i)}-\frac{1}{m_k}$ and $y(x_i)=f(x_i)-\mu_k$ for all $i\in\{1,\ldots,k\}$. Then $E_k(u)=0$ and $E_k(y)=0$, and therefore
\begin{align}\label{eq:expectationwithu}
E_q(f)
&=\sum_{i=1}^k q(x_i)f(x_i)
=E_k\left(m_k\frac{q}{p}f_k\right)\notag\\
&=E_k\left(\left(m_ku+1\right)(y+\mu_k)\right)=m_kE_k\left(u y\right)+m_k\mu_k E_k\left(u\right)+E_k\left(y\right)+\mu_k\notag\\
&=\mu_k+m_kE_k\left(u y\right).
\end{align}
and
\begin{align}\label{eq:chi2divergencewithu}
d_{\chi^2}(q,p)
&=\sum_{i=1}^n p(x_i)\left(\frac{q(x_i)}{p(x_i)}-1\right)^2
=m_k\sum_{i=1}^k p_k(x_i)\left(\frac{q(x_i)}{p(x_i)}-1\right)^2+\sum_{i=k+1}^n p(x_i)\notag\\
&=m_kE_k\left(\left(u+\frac{1}{m_k}-1\right)^2\right)+1-m_k\notag\\
&=m_kE_k\left(u^2\right)+m_k\left(\frac{1}{m_k}-1\right)^2+2m_k\left(\frac{1}{m_k}-1\right)E_k(u)+1-m_k\notag\\
&=m_kE_k\left(u^2\right)+m_k\left(\frac{1}{m_k}-1\right)^2+1-m_k,\notag\\
&=m_kE_k\left(u^2\right)+\frac{(1-m_k)^2+m_k(1-m_k)}{m_k}\notag\\
&=m_kE_k\left(u^2\right)+\frac{1-m_k}{m_k}.
\end{align}
This also implies that
\begin{align}\label{eq:chi2divergenceconstraintwithu}
d_{\chi^2}(q,p)\leq\delta
\Leftrightarrow
E_k(u^2)&\leq\frac{m_k\delta-(1-m_k)}{m_k^2}.
\end{align}
Due the Cauchy-Schwarz inequality, we also know that
\begin{equation}\label{eq:cauchy-schwarz}
\vert E_k(uy)\vert
\leq\sqrt{E_k(u^2)}\sqrt{E_k(y^2)}=\sigma_k\sqrt{E_k(u^2)},
\end{equation}
with equality if and only if $u$ and $y$ are linearly dependent. Since $y$ is not the zero function (because $f(x_{\ell+1})>f(x_\ell)$), this is the case if and only if $u=\alpha y$ for some $\alpha\in\mathbb{R}$, in which case we have that $E_k(y^2)=\sigma_k^2$, $E_k(u^2)=\alpha^2E_k(y^2)=\alpha^2\sigma_k^2$ and $E_k(uy)=\alpha E_k(y^2)=\alpha\sigma_k^2$.

Let us now start by proving the stated equivalences.

If $\mathcal{Q}_k\neq\emptyset$, there is some $q\in\mathcal{Q}_k$ and therefore, due to Equation~\eqref{eq:chi2divergenceconstraintwithu}, some corresponding $u$ such that $E_k(u^2)\leq(m_k\delta-(1-m_k))/m_k^2$. Since $m_k>0$ and $E_k(u^2)\geq0$, this implies that $m_k\delta-(1-m_k)\geq0$. 

Conversely, if $m_k\delta-(1-m_k)\geq0$, if we define $q\in\mathbb{R}^\mathcal{X}$ by $q(x_i)=0$ for all $i>k$ and $q(x_i)=\frac{p(x_i)}{m_k}$ for all $i\in\{1,\ldots,k\}$, then $u=0$ and $E_k(u^2)=0\leq(m_k\delta-(1-m_k))/m_k^2$, which implies that $d_{\chi^2}(q,p)\leq\delta$ due to Equation~\eqref{eq:chi2divergenceconstraintwithu}. So $q\in\mathcal{Q}_k$, which implies that $\mathcal{Q}_k\neq\emptyset$.

That $m_k\delta-(1-m_k)\geq0$ if and only if $\delta\geq(1-m_k)/m_k$ is a matter of simple verification.

For the remainder of the proof, we now assume that $\mathcal{Q}_k\neq\emptyset$, and therefore also that $m_k\delta-(1-m_k)\geq0$.

For any $q\in\mathcal{Q}_k$, it follows from Equations~\eqref{eq:expectationwithu}, \eqref{eq:cauchy-schwarz} and~\eqref{eq:chi2divergenceconstraintwithu} that
\begin{equation*}
E_q(f)=\mu_k+m_kE_k\left(u y\right)
\geq\mu_k-m_k\sigma_k\sqrt{E_k(u^2)}
\geq\mu_k-\sigma_k\sqrt{m_k\delta-(1-m_k)}.
\end{equation*}
Since $m_k>0$, the first inequality is furthermore strict unless $u=\alpha y$ for some $\alpha\leq0$ (and then $E_k(uy)=\alpha\sigma_k^2$ and $E_k(u^2)=\alpha^2\sigma_k^2$) and the second inequality is then strict unless $E_k(u^2)=(m_k\delta-(1-m_k))/m_k^2$. So the unique way to make both inequalities an equality, and hence to achieve the minimum in Equation~\eqref{eq:lemma:lowerexpectationchi2}, is if $u=\alpha y$ for
\begin{equation*}
\alpha=-\frac{1}{m_k\sigma_k}\sqrt{m_k\delta-(1-m_k)},
\end{equation*}
which implies that
\begin{equation*}
    q(x_i)=p(x_i)\left(\frac{1}{m_k}+\alpha y(x_i)\right)=\frac{p(x_i)}{m_k}\left(1-\frac{f(x_i)-\mu_k}{\sigma_k}\sqrt{m_k\delta-(1-m_k)}\right)
\end{equation*}
for all $i\in\{1,\ldots,k\}$. So $q_k$ is indeed the unique element of $\mathcal{Q}_k$ that attains the minimum in Equation~\eqref{eq:lemma:lowerexpectationchi2}, and this minimum is indeed equal to $\mu_k-\sigma_k\sqrt{m_k\delta-(1-m_k)}$.

That $q_k(x_k)>0$ if and only if $\delta<\delta_k$ follows from the following chain of equivalences:
\begin{align*}
q_k(x_k)>0
&\Leftrightarrow
1-\frac{f(x_k)-\mu_k}{\sigma_k}\sqrt{m_k\delta-(1-m_k)}>0\\
&\Leftrightarrow
\sqrt{m_k\delta-(1-m_k)}<\frac{\sigma_k}{f(x_k)-\mu_k}\\
&\Leftrightarrow
m_k\delta-(1-m_k)<\frac{\sigma_k^2}{(f(x_k)-\mu_k)^2}\\
&\Leftrightarrow
\delta<\frac{1}{m_k}\left(\frac{\sigma_k^2}{(f(x_k)-\mu_k)^2}+1-m_k\right)=\delta_k.
\end{align*}
That $q_k(x_k)=0$ if and only if $\delta=\delta_k$ can be shown by replacing the strict inequalities in the above chain of equivalences by equalities.

That $q_k(x_k)>0$ implies that $q_k(x_i)>0$ for all $i\in\{1,\ldots,k\}$, and similarly that $q_k(x_k)=0$ implies that $q_k(x_i)\geq0$ for all $i\in\{1,\ldots,k\}$, follows from Equation~\eqref{eq:lemma:lowerexpectationchi2:qk} and the fact that $f(x_i)\leq f(x_k)$ for all $i\in\{1,\ldots,k\}$. 

\end{proof}

\begin{proof}[Proof of Proposition~\ref{prop:orderingdeltak}]
Consider any $k\in\{\ell+1,\ldots,n\}$. Since $k>\ell$, it follows from the definition of $\ell$ that $f(x_k)\geq f(x_{\ell+1})>f(x_\ell)=f(x_1)$. This implies that the restriction of $f$ to $\{1,\ldots,k\}$ is not a constant function, and therefore that $\sigma_k^2>0$ and $f(x_k)-\mu_k>0$.
Since $p$ is strictly positive and $k>\ell\geq1$, we also have that $0<m_k\leq1$. Combining the above, we conclude that $\delta_k$ is well-defined and positive.

Now consider any $k\in\{\ell+1,\ldots,n-1\}$. We will show that $\delta_k\geq\delta_{k+1}$, which will then imply the desired ordering of the $\delta_k$. Assume \emph{ex absurdo} that $\delta_k<\delta_{k+1}$ and let $\delta=\delta_{k+1}$. Since $\delta=\delta_{k+1}\geq\frac{1-m_{k+1}}{m_{k+1}}$, Lemma~\ref{lemma:lowerexpectationchi2} tells us that there is a $q_{k+1}\in\mathcal{Q}_{k+1}$ such that $E_{q_{k+1}}(f)=\min_{q\in\mathcal{Q}_{k+1}}E_q(f)$ with $q_{k+1}(x_{k+1})=0$ and $q_{k+1}(x_k)\geq0$. Since $q_{k+1}(x_{k+1})=0$, we also have that $q_{k+1}\in\mathcal{Q}_k$, so $\mathcal{Q}_k\neq\emptyset$. Since $\delta_k<\delta_{k+1}=\delta$, Lemma~\ref{lemma:lowerexpectationchi2} therefore also tells us that there is a unique $q_k\in\mathcal{Q}_k$ such that $E_{q_k}(f)=\min_{q\in\mathcal{Q}_k} E_q(f)$, with $q_k(x_k)<0$. Since $q_{k+1}(x_k)\geq0$, we have that $q_{k+1}\neq q_k$. Since $q_k$ was the unique element of $\mathcal{Q}_k$ such that $E_{q_k}(f)=\min_{q\in\mathcal{Q}_k} E_q(f)$, this implies that $E_{q_{k+1}}(f)>\min_{q\in\mathcal{Q}_k} E_q(f)$. But since $\mathcal{Q}_k\subseteq\mathcal{Q}_{k+1}$, we also know that $\min_{q\in\mathcal{Q}_{k+1}} E_q(f)\leq\min_{q\in\mathcal{Q}_k} E_q(f)$. Putting everything together, we arrive at the contradiction
\begin{equation*}
E_{q_{k+1}}(f)=\min_{q\in\mathcal{Q}_{k+1}} E_q(f)\leq\min_{q\in\mathcal{Q}_k} E_q(f)<E_{q_{k+1}}(f).
\end{equation*}
So indeed $\delta_k\geq\delta_{k+1}$.
\end{proof}

\begin{lemma}\label{lemma:movemasschi2}
Consider any $q\in B_{\chi^2}(p,\delta)$ and any $i,j\in\{1,\ldots,n\}$ such that $q(x_i)=0$ and $q(x_j)>0$. Then there is some $0<\epsilon<q(x_j)$ such that the $q'\in\Sigma_\mathcal{X}$ defined by 
\begin{equation*}
q'(x_k)=
\begin{cases}
\epsilon & \text{if } k=i,\\
q(x_j)-\epsilon & \text{if } k=j,\\
q(x_k) & \text{if } k\in\{1,\ldots,n\}\setminus\{i,j\}
\end{cases}
\end{equation*}
is also an element of $B_{\chi^2}(p,\delta)$. The corresponding expectation is given by $E_{q'}(f)=E_q(f)+\epsilon(f(x_i)-f(x_j))$.
\end{lemma}

\begin{proof}
Consider any $0<\epsilon<q(x_j)$. Then $q'$ is clearly a probability mass function. We also find that
\begin{align*}
d_{\chi^2}(q',p)
&=d_{\chi^2}(q,p)+\frac{(q'(x_i)-p(x_i))^2}{p(x_i)}-\frac{(q(x_i)-p(x_i))^2}{p(x_i)}\\
&~~~~~~~~~~~~~~~~~~~~~~~~~~~~~+\frac{(q'(x_j)-p(x_j))^2}{p(x_j)}-\frac{(q(x_j)-p(x_j))^2}{p(x_j)}\\
&=d_{\chi^2}(q,p)+\frac{(\epsilon-p(x_i))^2}{p(x_i)}-\frac{p(x_i)^2}{p(x_i)}\\
&~~~~~~~~~~~~~~~~~~~~~~~~~~~~~+\frac{(q(x_j)-\epsilon-p(x_j))^2}{p(x_j)}-\frac{(q(x_j)-p(x_j))^2}{p(x_j)}\\
&=d_{\chi^2}(q,p)+\frac{\epsilon^2-2\epsilon p(x_i)}{p(x_i)}+\frac{\epsilon^2-2\epsilon(q(x_j)-p(x_j))}{p(x_j)}\\
&=d_{\chi^2}(q,p)+\epsilon\left(\epsilon\left(\frac{1}{p(x_i)}+\frac{1}{p(x_j)}\right)-2\frac{q(x_j)}{p(x_j)}\right).
\end{align*}
Since $q(x_j)>0$, we can choose $\epsilon$ small enough such that the second term is strictly negative. Since $d_{\chi^2}(q,p)\leq\delta$, this then implies that $d_{\chi^2}(q',p)<\delta$, so indeed $q'\in B_{\chi^2}(p,\delta)$.

The corresponding expectation is given by
\begin{align*}
E_{q'}(f)&=\sum_{k=1}^n q'(x_k)f(x_k)\\
&=\sum_{k=1}^n q(x_k)f(x_k)+\epsilon(f(x_i)-f(x_j))
=E_q(f)+\epsilon(f(x_i)-f(x_j)).
\end{align*}
\end{proof}

\begin{lemma}\label{lemma:maxsupportchi2}
There are $q^*\in B_{\chi^2}(p,\delta)$ and $k^*\in\{\ell,\ldots,n\}$ such that $q^*(x_i)>0$ for all $i\leq k^*$, $q^*(x_i)=0$ for all $i>k^*$ and $E_{q^*}(f)=\underline{E}^{\chi^2}_{p,\delta}(f)$.
\end{lemma}
\begin{proof}
Since $\smash{B_{\chi^2}(p,\delta)}$ is closed, bounded and non-empty, there is some $q\in B_{\chi^2}(p,\delta)$ such that $E_q(f)=\underline{E}^{\chi^2}_{p,\delta}(f)$. Among all such $q$, pick one that has maximal support size $\vert\{i\in\{1,\ldots,n\}:q(x_i)>0\}\vert$ and denote it by $q^*$, and let $k^*$ be the largest index such that $q^*(x_{k^*})>0$. Then by construction, $q^*(x_i)=0$ for all $i>k^*$.
We will show that $k^*\geq\ell$ and that $q(x_i)>0$ for all $i<k^*$.

Suppose that $k^*<\ell$. Then $f(x_{k^*})=f(x_\ell)$.
Applying Lemma~\ref{lemma:movemasschi2} with $i=\ell$ and $j=k^*$, we find that there is some $0<\epsilon<q^*(x_{k^*})$ such that the corresponding $q'\in\Sigma_\mathcal{X}$ is also an element of $B_{\chi^2}(p,\delta)$ and satisfies
\begin{equation*}
E_{q'}(f)=E_{q^*}(f)+\epsilon(f(x_\ell)-f(x_{k^*}))=E_{q^*}(f),
\end{equation*}
which contradicts the definition of $q^*$, since $q'$ has a strictly larger support size than $q^*$. So indeed $k^*\geq\ell$. 

Now suppose that there is some $i<k^*$ such that $q^*(x_i)=0$. Then we can apply Lemma~\ref{lemma:movemasschi2} with this $i$ and $j=k^*$ to find that there is some $0<\epsilon<q^*(x_{k^*})$ such that the corresponding $q'\in\Sigma_\mathcal{X}$ is also an element of $B_{\chi^2}(p,\delta)$ and satisfies
\begin{equation*}
E_{q'}(f)=E_{q^*}(f)+\epsilon(f(x_i)-f(x_{k^*}))\leq E_{q^*}(f),
\end{equation*}
which again contradicts the definition of $q^*$ because $q'$ has a strictly larger support size than $q^*$. So indeed $q^*(x_i)>0$ for all $i<k^*$.
\end{proof}

\begin{proof}[Proof of Theorem~\ref{thm:lowerexpectationchi2}]
Let $q^*\in B_{\chi^2}(p,\delta)$ and $k^*\in\{\ell,\ldots,n\}$ be as in Lemma~\ref{lemma:maxsupportchi2}. So $q^*(x_i)>0$ for all $i\leq k^*$, $q^*(x_i)=0$ for all $i>k^*$ and $\smash{E_{q^*}(f)=\underline{E}^{\chi^2}_{p,\delta}(f)}$.

We now consider two cases. If $r=\ell$, then $k^*=\ell$, so $f(x_1)=\ldots=f(x_\ell)$ and therefore also $\mu_r=f(x_1)$. Since $q^*(x_i)=0$ for all $i>\ell=r$, this implies that
\begin{align*}
    \underline{E}^{\chi^2}_{p,\delta}(f)=
E_{q^*}(f)&=\sum_{i=1}^\ell q^*(x_i)f(x_i)=f(x_1)\sum_{i=1}^\ell q^*(x_i)=f(x_1),
\end{align*}
as required. So it remains to consider the case that $r>\ell$.

In that case, we first show that $r\geq k^*$. 
To that end, assume \emph{ex absurdo} that $r<k^*$. Then 
\begin{equation*}
    \delta\geq\delta_{k^*}\geq\frac{1-m_{k^*}}{m_{k^*}}.
\end{equation*}
Since $k^*>r\geq\ell$, it therefore follows from Lemma~\ref{lemma:lowerexpectationchi2} that $q_{k^*}$ is the unique element of $\mathcal{Q}_{k^*}$ that attains the minimum in Equation~\eqref{eq:lemma:lowerexpectationchi2} for $k=k^*$ and that $q_{k^*}(x_{k^*})\leq0$. Since $q^*\in\mathcal{Q}_{k^*}$ and $q^*(x_{k^*})>0$, this implies that $E_{q^*}(f)>E_{q_{k^*}}(f)$.

Since $q^*(x_i)>0$ for all $i\leq k^*$, there is some $\lambda\in(0,1)$ such that $q'=(1-\lambda)q^*+\lambda q_{k^*}$ is an element $\Sigma_\mathcal{X}$. Due to Lemma~\ref{lemma:convexitychi2}, $q'$ is in fact an element of $B_{\chi^2}(p,\delta)$, so $E_{q'}(f)\geq\underline{E}^{\chi^2}_{p,\delta}(f)=E_{q^*}(f)$. On the other hand, we also have that \begin{equation*}
    E_{q'}(f)
    =(1-\lambda)E_{q^*}(f)+\lambda E_{q_{k^*}}(f)
    <(1-\lambda)E_{q^*}(f)+\lambda E_{q^*}(f)
     =E_{q^*}(f),
\end{equation*}which is a contradiction. So it must be that $r\geq k^*$.

Since $r\geq k^*$, we know that $q^*\in\mathcal{Q}_r$, so $\mathcal{Q}_r\neq\emptyset$. Since $r>\ell$, we therefore know from Lemma~\ref{lemma:lowerexpectationchi2} that
\begin{equation*}
\mu_r-\sigma_r\sqrt{m_r\delta-(1-m_r)}
=\min_{q\in\mathcal{Q}_r} E_q(f)
=E_{q_r}(f)
\leq E_{q^*}(f)=\underline{E}^{\chi^2}_{p,\delta}(f).
\end{equation*}
On the other hand, since $\delta<\delta_r$, we also know from Lemma~\ref{lemma:lowerexpectationchi2} that $q_r(x_i)>0$ for all $i\in\{1,\ldots,r\}$, which implies that $q_r\in B_{\chi^2}(p,\delta)$ and therefore, that
\begin{equation*}
\underline{E}^{\chi^2}_{p,\delta}(f)\leq E_{q_r}(f).
\end{equation*}
Combining both inequalities, we obtain Equation~\eqref{eq:lowerexpectationchi2}.
\end{proof}

\begin{proof}[Proof of Equation~\eqref{eq:lowerexpectationchi2three}]
It suffices to observe that the three cases correspond to $r=3$, $r=2$ and $r=1$, respectively, and to then apply Theorem~\ref{thm:lowerexpectationchi2} taking into account that $m_3=1$.
\end{proof}

\begin{proof}[Proof of Equation~\eqref{eq:lowerexpectationchi2two}]
If $f(x_1)=f(x_2)$, then $\ell=n=r$, so we know from Theorem~\ref{thm:lowerexpectationchi2} that $\underline{E}^{\chi^2}_{p,\delta}(f)=f(x_1)$. Since then also $\sum_{i=1}^2p(x_i)f(x_i)=f(x_1)$ and $\vert f(x_2)-f(x_1)\vert=0$, this is consistent with the above expression. 

If $f(x_1)<f(x_2)$, then $\ell=1$, $m_2=1$, $\mu_2=\sum_{i=1}^2p(x_i)f(x_i)$, 
\begin{align*}
\sigma_2^2
&=p(x_1)[f(x_1)-\mu_2]^2+p(x_2)[f(x_2)-\mu_2]^2\\
&=p(x_1)\left[p(x_2)(f(x_1)-f(x_2))\right]^2+p(x_2)\left[p(x_1)(f(x_2)-f(x_1))\right]^2\\
&=p(x_1)p(x_2)(f(x_2)-f(x_1))^2(p(x_1)+p(x_2))=p(x_1)p(x_2)(f(x_2)-f(x_1))^2
\end{align*}
and
\begin{equation*}
    \delta_2=\frac{1}{m_2}\left(\frac{\sigma_2^2}{(f(x_2)-\mu_2)^2}+1-m_2\right)=\frac{p(x_1)p(x_2)(f(x_2)-f(x_1))^2}{p(x_1)^2(f(x_2)-f(x_1))^2}=\frac{p(x_2)}{p(x_1)}.
\end{equation*}
If $\delta\geq\frac{p(x_2)}{p(x_1)}=\delta_2$, then $r=1=\ell$, so it follows from Theorem~\ref{thm:lowerexpectationchi2} that $\underline{E}^{\chi^2}_{p,\delta}(f)=f(x_1)$, which is consistent with the above expression. If $\delta<\frac{p(x_2)}{p(x_1)}=\delta_2$, then $r=2$, so it follows from Theorem~\ref{thm:lowerexpectationchi2} that
\begin{equation*}
\underline{E}^{\chi^2}_{p,\delta}(f)
=\mu_2-\sigma_2\sqrt{m_2\delta-(1-m_2)},
\end{equation*}
which is again consistent with the above expression.
\end{proof}

\end{document}